%% file: non-ori-3d-4d.tex
\documentclass[11pt]{amsart}

\usepackage{graphicx}
\usepackage[usenames,dvipsnames]{xcolor}
\usepackage{tikz-cd}
\usepackage{pinlabel}
\usepackage[pdftex,%
  final,%
  colorlinks=true,%
  linkcolor=NavyBlue,%
  citecolor=NavyBlue,%
  filecolor=NavyBlue,%
  menucolor=NavyBlue,%
  urlcolor=NavyBlue,%
  bookmarks=true,%
  bookmarksdepth=3,%
  bookmarksnumbered=true,%
  bookmarksopen=true,%
  bookmarksopenlevel=2,%
]{hyperref}

\newcommand{\dfn}[1]{{\textbf{#1}}}
\numberwithin{equation}{section}

\newcommand{\rr}{\ensuremath{\mathbb{R}}}
\newcommand{\zz}{\ensuremath{\mathbb{Z}}}
\newcommand{\qq}{\ensuremath{\mathbb{Q}}}

\newcommand{\nn}{\ensuremath{\mathbb{N}}}

\theoremstyle{plain}
\newtheorem{thm}{Theorem}[section]
\newtheorem{cor}[thm]{Corollary}
\newtheorem{lem}[thm]{Lemma}

\newtheorem{prop}[thm]{Proposition}
\newtheorem{conj}[thm]{Conjecture}

\theoremstyle{definition}
\newtheorem{defn}[thm]{Definition}

\theoremstyle{remark}
\newtheorem{rem}[thm]{Remark}
\newtheorem{ex}[thm]{Example}

\newcommand{\sg}{\ensuremath{d}}
\newcommand{\asg}{\ensuremath{\overline{\gamma}}}
\newcommand{\spanning}{\ensuremath{\mathcal{F}}}
\newcommand{\candidate}{\ensuremath{\mathcal{C}}}
\newcommand{\ep}{\ensuremath{\Lambda}}
\newcommand{\bep}{\ensuremath{\tilde{\Lambda}}}
\DeclareMathOperator{\lk}{lk}

\DeclareMathOperator{\wri}{wr}

\DeclareMathOperator{\twist}{tw}

\begin{document}

\title[A Refinement of the Spanning Surface Defect]{A Refinement of the Spanning Surface Defect in $3$ and $4$ Dimensions}

\author[J. Knihs]{Julia Knihs} \address{Haverford College,
Haverford, PA 19041} \email{jknihs0401@gmail.com} 

\author[J. Patel]{Jeanette Patel} \address{Haverford College,
Haverford, PA 19041} \email{jpatel@haverford.edu} 

\author[T. Rugg]{Thea Rugg} \address{Cornell University, Ithaca, NY 14853} \email{dgr77@cornell.edu}

\author[J. Sabloff]{Joshua M. Sabloff} \address{Haverford College,
Haverford, PA 19041} \email{jsabloff@haverford.edu} 

\date{\today}

\begin{abstract}
The spanning surface defect uses spanning surfaces of a knot in the $3$-sphere to measure how far a knot is from being alternating.  We refine the spanning surface defect and extend the definition to take into account surfaces in the $4$-ball.  We use these extensions to make comparisons between the $3$- and $4$-dimensional settings, to reframe non-orientable slice-torus bounds on the non-orientable $4$-genus, and to prove a connected sum formula.
\end{abstract}

\maketitle

% ****************************************

\setcounter{tocdepth}{1}
\tableofcontents

% ********************
\section{Introduction}
\label{sec:intro}

\input{intro}

% ********************
\section{Spanning Surfaces in $3$ Dimensions}
\label{sec:geography3}

\input{geography3}

% ********************
\section{Fillings in $4$ Dimensions}
\label{sec:geography4}

\input{geography4}

% ********************
\section{Motivation from Torus Knots}
\label{sec:torus}

\input{torus}

% ********************
\section{Ordinary Genus Gaps for Pretzel Knots}
\label{sec:gap}

\input{gap}

% ********************
\section{Spanning Surface Defect Gaps for Slice Pretzel Knots}
\label{sec:stable-gap-pretzel}

\input{stable-gap-pretzel}

% ********************
\section{Spanning Surface Defect Gaps via Connected Sums}
\label{sec:stable-gap-sum}

\input{stable-gap-sum}

% ****************************************
\bibliographystyle{amsplain} 
\bibliography{main}

\end{document}

%% file: intro.tex
%!TEX root = non-ori-3d-4d.tex

Surfaces that span a knot in the $3$-sphere --- that is, compact embedded surfaces whose boundary is a given knot --- are important tools in understanding the topology of the complement of a knot.  For example, one might seek to understand the minimal genus of an orientable --- or non-orientable --- spanning surface of a knot, or to enumerate the set of possible slopes that spanning surfaces imprint on the boundary of a tubular neighborhood of a knot, or even to characterize whether a knot has an alternating diagram using particularly nice spanning surfaces \cite{greene:alternating, howie:alternating}.  In extending the ideas of Greene's and Howie's characterizations of alternating knots to a wider class of knots, Ito \cite{ito:almost-alternating} defined the spanning surface defect to measure how far a knot is from being alternating.  The goal of this paper is to refine this measurement and to extend it to $4$ dimensions.

%%%%%
\subsection{Spanning Surfaces, Fillings, and the Defect}

To define and refine the spanning surface defect of a knot $K \subset S^3$, we begin by setting notation.  A compact smooth embedded surface $F \subset S^3$ with $\partial F = K$ is a \dfn{spanning surface} of $K$.  Let $\spanning_3(K)$ be the set of all spanning surfaces of $K$. Define the \dfn{non-orientable $3$-genus}, or \dfn{crosscap number}, of $K$ to be
\[\gamma_3(K) = \min\{b_1(F)\,:\, F \in \spanning_3(K)\}.\]
This quantity was first introduced by Clark \cite{clark:crosscap}, and has been computed for torus knots \cite{teragaito:torus}, $2$-bridge knots \cite{ht:2-bridge-crosscap}, pretzel knots \cite{im:pretzel} and others; further, there exists an algorithm for determining $\gamma_3$ for alternating knots \cite{ak:alternating} and more general algorithms that use normal surface theory \cite{bo:crosscap-compute}.

Similarly, a compact smooth properly embedded surface $F \subset B^4$ with $\partial F = K$ is a \dfn{filling} of $K$. Let $\spanning_4(K)$ be the set of all fillings of $K$. The \dfn{non-orientable $4$-genus} of a knot $K$ is
\[  \gamma_4(K) = \min\{b_1(F)\,:\, F \in \spanning_4(K)\}.\]
Note that $\gamma_4(K) \leq \gamma_3(K)$. The non-orientable $4$-genus was first studied by Viro \cite{viro:positioning} and Yasuhara \cite{yasuhara:connecting}.  Efforts at understanding $\gamma_4$ use tools from algebraic topology \cite{gl:non-ori-4-genus, gl:signature} and from modern invariants such as Heegaard Floer homology \cite{allen:geography, batson:non-ori-slice, bkst:non-ori-genus, gm:non-ori-genus, oss:unoriented}, Khovanov homology \cite{ballinger:concordance-kh}, and gauge theory \cite{ds:cs-clasp}; see \cite{sato:unori-slice-torus} and Section~\ref{sec:geography4} below for a framework for many modern invariants.

A non-orientable filling $F$ has a second interesting homological invariant beyond its first Betti number, namely the Euler class of its normal bundle, termed  the \textbf{normal Euler number} $e(F) \in 2\zz$.  The normal Euler number of a spanning surface is defined to be the normal Euler number of the filling that results from pushing the interior of the spanning surface into $B^4$; note that the normal Euler number of a spanning surface is the \emph{negative} of its \dfn{boundary slope}. See Sections~\ref{sec:geography3} and \ref{sec:geography4} for the precise definitions of the normal Euler number, boundary slope, and other foundational concepts.

We are now ready to define our central object of study, the spanning surface defect.  In fact, we will define several flavors of the defect.  First, we normalize the genus of a spanning surface or filling $F$ of $K$ using the normal Euler number of $F$ and the signature of $K$.  In particular, for a spanning surface or filling $F$ of $K$, define its \dfn{Euler-normalized first Betti numbers} to be 
\[\Gamma^{\pm}(F) = b_1(F) \pm \left(\sigma(K)-\frac{1}{2}e(F)  \right),\]
where $\sigma(K)$ is the signature of $K$.  Note that the quantity $\sigma(K) - \frac{1}{2}e(F)$ is the signature of the Gordon-Litherland form of $F$; see Section~\ref{ssec:signature}, below.

A key feature of $\Gamma^+(F)$ (resp.\ $\Gamma^-(F)$) is that it is invariant under the \dfn{addition of a $+$-twisted band} (resp.\ a $-$-twisted band), as in Figure~\ref{fig:twisted-band}.  Such a band addition raises the first Betti number by $1$ and changes the normal Euler number by $\pm 2$, while leaving the isotopy class of the boundary unchanged.  We denote by $\twist_{\pm}(F)$ the addition of a $\pm$-twisted band to $F$.  This feature of $\Gamma^\pm(F)$ is the origin of the moniker ``Euler-normalized''.

\begin{figure}
\labellist
\small\hair 2pt
 \pinlabel {$\twist_+$} [l] at 114 55
 \pinlabel {$\twist_-$} [r] at 191 55
\endlabellist

\begin{center}
\includegraphics{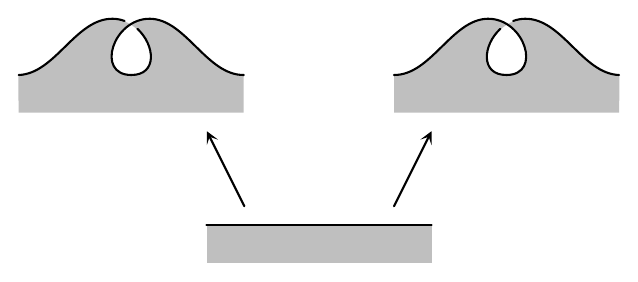}
\caption{The addition of $\pm$-twisted band to a spanning surface or filling $F$ of $K$ raises the first Betti number by $1$ and changes the normal Euler number by $\pm 2$, but leaves invariant $\Gamma^\pm(F)$ and the isotopy class of the boundary. The sign convention for the normal Euler number means that the sign of the crossing is the opposite of the sign of the twist.}
\label{fig:twisted-band}
\end{center}
\end{figure}

As we shall see in Section~\ref{ssec:signature}, a second important feature of the Euler-normalized first Betti number is that it is non-negative, and hence it makes sense to minimize over all spanning surfaces or fillings.

\begin{defn} \label{defn:normalized-genus}
	The \dfn{$\pm$-spanning surface defect in dimension $n$} of a knot $K$, for $n=3,4$, is defined to be 
	\[\sg^\pm_n(K) = \min\{\Gamma^\pm(F)\;:\; F \in \spanning_n(K)\}.\]
\end{defn}

It is straightforward to check that $\frac{1}{2}(\sg_3^+(K)+\sg_3^-(K))$ coincides with Ito's original definition of the spanning surface defect \cite{ito:almost-alternating}:
\[\sg_3(K) = \frac{1}{2} \min \left\{b_1(F) + b_1(G) -  |\sigma_{gl}(F) - \sigma_{gl}(G)| \;:\; F,G \in \spanning_3(K)\right\},\]
where $\sigma_{gl}$ is the signature of the Gordon-Litherland form of a spanning surface.  Ito's motivation for defining $d_3$ is encapsulated by the following theorem, which we refine slightly in light of Definition~\ref{defn:normalized-genus}.

\begin{thm}[\cite{greene:alternating, howie:alternating} via \cite{ito:almost-alternating}] \label{thm:alternating}
A knot $K$ is alternating if and only if $\sg^+_3(K) = 0 = \sg^-_3(K)$.
\end{thm}

Ito further shows that the spanning surface defect is a lower bound on the alternating genus of a knot, and hence on the Turaev genus.  The Turaev genus $g_T$ was developed by Turaev to study the Jones polynomial and the Tait conjectures, and was formalized in \cite{dfkls:jones}.  As in Theorem~\ref{thm:alternating}, the Turaev genus of $K$ vanishes if and only if $K$ is alternating.  We obtain the following bound.

\begin{thm}[\cite{ito:almost-alternating} and \cite{lowrance:alt-dist}] \label{thm:turaev}	
	For any knot $K$, $\sg_3(K) \leq g_T(K)$.
\end{thm}

In fact, since the Turaev genus is realized on an adequate diagram \cite{kalfagianni:jones-slopes}, we make the following conjecture.

\begin{conj} \label{conj:adequate-turaev}
	For an adequate knot $K$, $g_T(K) = \sg_3(K)$.
\end{conj}

%%%%%
\subsection{Inspiration from Geography Questions}
\label{ssec:geography-inspiration}

Our exploration of the $\pm$-spanning surface defect arose from Allen's geography problem for the non-orientable $4$-genus \cite{allen:geography}; Howie has also taken this perspective for spanning surfaces in $S^3$ \cite{howie:pc}.  For a knot $K$, Allen asked for a description of the set $R_4(K)$ of possible pairs $(e, b)$ realized by the normal Euler number and first Betti number, respectively, of a filling of $K$.  We may define $R_3(K)$ similarly; note that $R_3(K) \subset R_4(K)$.  As depicted in Figure~\ref{fig:intro-geography}, the set $R_*(K)$ is constrained by the Gordon-Litherland signature bound, with points of $R(K)$ lying in the wedge \[W_\sigma = \left\{b \geq \frac{1}{2}|e-2\sigma(K)|\right\} \subset \zz^2.\] 
In the geography diagram, we may read off $\gamma_n(K)$ as the lowest height of a point in $R_n(K)$, whereas $\sg^+_n(K)$ (resp.\ $\sg^-_n(K)$) is the minimum vertical distance from $R_n(K)$ to the upward-sloping (resp.\ downward-sloping) boundary of $W_{\sigma}$.  

\begin{figure}
\labellist
\small\hair 2pt
 \pinlabel {$2\sigma$} [t] at 62 6
 \pinlabel {$e$} [l] at 193 9
 \pinlabel {$b$} [b] at 170 84
 \pinlabel {$\Gamma^-(F)$} [l] at 30 44
 \pinlabel {$\Gamma^+(F')$} [r] at 116 55
 \pinlabel {$F$} [r] at 24 62
 \pinlabel {$F'$} [l] at 120 73
\endlabellist

\begin{center}
\includegraphics{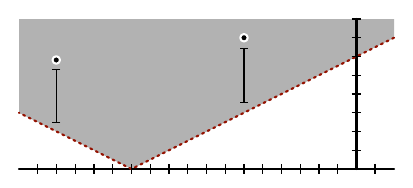}
\caption{The geography $R_*(K)$ is constrained by the Gordon-Litherland signature bound.  The $\pm$-spanning surface defect measures the vertical distance from the $(e,b)$ coordinates of a surface to the signature bound.}
\label{fig:intro-geography}
\end{center}
\end{figure}

The discussion above suggests that it may be profitable to think about the $\pm$-spanning surface defects as a type of non-orientable genus normalized by the signature and normal Euler number.  For example, we may reinterpret the Heegaard-Floer $\upsilon$ bound on the non-orientable $4$-genus of \cite{oss:unoriented} as in the following proposition, which we shall prove in Section~\ref{ssec:bound}:

\begin{prop} \label{prop:oss-sg}
	For any knot $K$, we have
	\begin{align*}
		\max\left\{\sigma(K) - 2\upsilon(K),0\right\} &\leq \sg^+_4(K),  \\
	\max\left\{2\upsilon(K) - \sigma(K),0\right\} &\leq \sg^-_4(K). 
	\end{align*}
\end{prop}

%%%%%
\subsection{Genus Gaps}

As in the orientable case, it is interesting to compare the non-orientable $3$- and $4$-genera of a knot.  We first address the  question of how large the gap between $\gamma_3(K)$ and $\gamma_4(K)$ can be.  Jabuka and Van Cott proved that the gap $\gamma_3(K) - \gamma_4(K)$ may be arbitrarily large for families of torus knots \cite{jvc:3vs4}.  We show that the gap is also arbitrarily large for a family of pretzel knots.

\begin{prop} \label{prop:pretzel-gap}
For any $n \geq 1$ and positive odd integers $k,r, p_1, \ldots, p_{n-1}$, the pretzel knot $P = P(-k,r,-r-1, p_1, -p_1, \ldots,p_{n-1}, -p_{n-1})$ satisfies
\begin{enumerate}
	\item $\gamma_4(P) = 1$,
	\item $\gamma_3(P) = 2n$, and 
	\item $g_4(P) \geq \frac{1}{2}(k - r - 2)$.
\end{enumerate}
\end{prop}

The gap between the spanning surface defects $\sg^\pm_3(P)$ and $\sg^\pm_4(P)$, on the other hand, is at most $2$, though it is unclear if it is exactly $2$.   To find a non-trivial gap between spanning surface defects, we instead examine pretzel knots of the form $P_n = P(-3, 3, n)$ for odd $n \geq 3$. Knots of this form are slice \cite{gj:pretzel-slice}. Hence, we obtain $\sg^\pm_4(P_n) = 0$.  In contrast, the  negative Euler-normalized non-orientable $3$-genus of $P_n$ is non-vanishing.

\begin{prop}\label{prop:pretzelside}
For all odd $n \geq 3$, we have $\sg_3^-(P(-3,3,n)) = 2$, while $\sg_3^+(P(-3,3,n))=0$.
\end{prop} 

Though the result in Proposition~\ref{prop:pretzelside} could have been derived from Theorem~\ref{thm:alternating}, Theorem~\ref{thm:turaev}, and the fact that non-alternating pretzel knots have Turaev genus $1$ \cite{ak:alternating}, the ideas underlying the calculation, which include enumerating essential surfaces in the exterior of $P_n$, demonstrate what may be needed to prove Conjecture~\ref{conj:adequate-turaev}.

We end by showing that the gap between the Euler-normalized non-orientable $3$- and $4$-genera may be arbitrarily large.  Specifically, we show that the Euler-normalized non-orientable $3$-genus is additive under connected sum (see Lemma~\ref{lem:connectedsum}); we note that Josh Howie independently proved this proposition \cite{howie:pc}. We then use the connected sum formula to prove:

\begin{thm} \label{thm:large-gap}
	For all $n \in \nn$, there exists a knot $K_n$ such that 
	\[\sg^-_3(K_n) - \sg^-_4(K_n) = 2n.\]
\end{thm}

The same result for $\sg^+$ holds by taking mirror images.

%%%%%
\subsection{Plan of the Paper}

The remainder of the paper is organized as follows: we  review non-orientable spanning surfaces and their geography in Section~\ref{sec:geography3}.   A parallel discussion of non-orientable fillings and their geography appears in Section~\ref{sec:geography4}, culminating in a proof of Proposition~\ref{prop:oss-sg}.  Next, we examine situations in which $\gamma_3(K) > \gamma_4(K)$ (Section~\ref{sec:gap}, which includes a proof of Proposition~\ref{prop:pretzel-gap}) and where $\sg^\pm_3 (K) > \sg^\pm_4 (K)$ (Section~\ref{sec:stable-gap-pretzel}, which includes a proof of Proposition~\ref{prop:pretzelside}). We end in Section~\ref{sec:stable-gap-sum} with a discussion of the behavior of $\sg^\pm_3$ under connected sum, which yields a proof of Theorem~\ref{thm:large-gap}.

\subsection*{Acknowledgements}

We thank David Futer and Josh Howie for enlightening discussions about the context and potential implications of this work.  We further thank the participants of the Philadelphia Area Contact / Topology (PACT) Seminar for their feedback on a presentation of work in progress.

%% file: geography3.tex
%!TEX root = non-ori-3d-4d.tex

The goal of this section is to describe the framework necessary to investigate the non-orientable spanning surfaces of a knot.   We start by reviewing the boundary slope of a spanning surface and by adapting Allen's geography problem to the $3$-dimensional setting; again, Howie has also taken this perspective.  We finish with a discussion of essential surfaces, their boundary slopes, and their relationship to the $3$-dimensional geography of a knot in Section~\ref{ssec:essential}.

% *****
\subsection{Boundary Slope and 3-Dimensional Geography}
\label{ssec:bs}

A possibly non-orientable spanning surface $F$ of a knot $K \subset S^3$ has two natural homological invariants:  its \dfn{first Betti number} $b_1(F)$ and its \dfn{boundary slope} $s(F)$.  To define the boundary slope, we push $K$ by a small amount along $F$ to obtain a knot $K'$; the boundary slope is then given by 
\[s(F) = \lk (K,K').\]
For later comparison with fillings in four dimensions, we also define the \dfn{($3$-dimensional) normal Euler number} of $F$ to be $e_3(F) = -s(F)$. See Figure~\ref{fig:trefoil-mobius} for an example of a trefoil knot bounding a M\"obius strip $F$ with $e_3(F)=-6$.

\begin{figure}
\begin{center}
\includegraphics{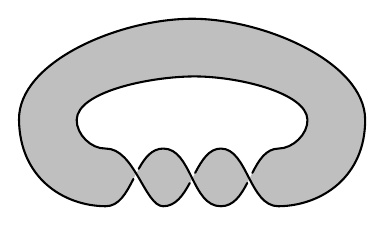}
\caption{A trefoil knot $K$ bounding a M\"obius strip $F$ with $e_3(F) = -6$.  Each crossing of $K$ induces four crossings (all positive) between $K$ and the pushoff $K'$.}
\label{fig:trefoil-mobius}
\end{center}
\end{figure}

The boundary slope has several useful properties. First, the boundary slope vanishes on orientable surfaces.  Second, the boundary slope is always even \cite[Proposition 2.2]{ak:alternating}.  In fact, the proof of Proposition 2.2 in \cite{ak:alternating} shows that 
\begin{equation} \label{eq:s-parity}
	e_3(F) \equiv s(F) \equiv 2b_1(F) \mod 4.
\end{equation}

We organize our exploration of spanning surfaces of a knot $K$ by adapting Allen's geography question for fillings in $4$ dimensions to the $3$-dimensional setting.

\begin{defn} [Adapted from \cite{allen:geography}] \label{defn:3dgeography} 
The \dfn{three-dimensional geography of a knot $K$} is the subset $R_3(K) \subset 2\zz \times \nn$ of all pairs $(e_3(F), b_1(F))$ for $F \in \spanning_3(K)$.  An element $(e,b) \in R_3(K)$ is called \dfn{realizable}.
\end{defn}

We typically depict the geography of a knot on a $2\zz \times \nn$ grid.  For example, Figure~\ref{fig:trefoil-geography-3} shows the $3$-dimensional geography of the trefoil knot; we will discuss why the figure is correct at the end of this subsection.

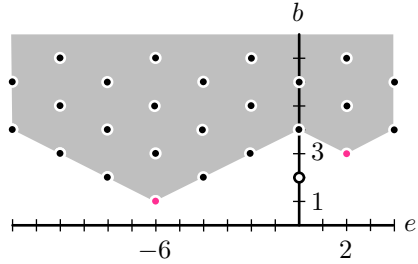
\begin{figure}
\labellist
\small\hair 2pt
 \pinlabel {$e$} [l] at 155 9
 \pinlabel {$b$} [b] at 117 84
 \pinlabel {$-6$} [t] at 63 5
 \pinlabel {$2$} [t] at 135 5
 \pinlabel {$1$} [l] at 120 18
 \pinlabel {$3$} [l] at 120 37
\endlabellist
\begin{center}
\includegraphics{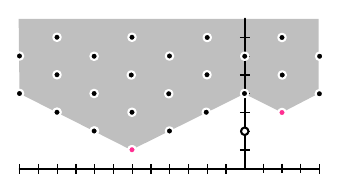}
\caption{The three-dimensional geography $R_3(K)$ for the trefoil knot $K$, with each solid dot denoting a point realized by a spanning surface.  The vertex at $(-6,1)$ is the M\"obius strip in Figure~\ref{fig:trefoil-mobius}, while the vertex $(2,3)$ is the positive twist of the Seifert surface.  The Seifert surface itself is denoted by an open dot at $(0,2)$. }
\label{fig:trefoil-geography-3}
\end{center}
\end{figure}

The addition of a twisted band is an important operation in the study of the $3$-dimensional geography.  Note once again that we define the sign of a twisted band in terms of the normal Euler number, not the boundary slope. The twisted band operation motivates the definition of a \dfn{wedge} $W_{(E,B)} \subset 2\zz \times \nn$, which consists of pairs $(e,b)$ that satisfy Equation~\eqref{eq:s-parity} and the relation 
\[\frac{1}{2}|e-E| \leq b-B.\] 
If $F$ is a spanning surface, we write $W_F$ for $W_{(e_3(F), b_1(F))}$.  The addition of twisted bands shows that if $F$ is a spanning surface for $K$, then $W_F \subset R_3(K)$.  In fact, we have
\begin{equation} \label{eq:wedge-geography}
	R_3(K) = \bigcup_{F \in \spanning_3(K)} W_F.
\end{equation}

To see that Figure~\ref{fig:trefoil-geography-3} correctly depicts the $3$-dimensional geography of the trefoil knot, we invoke Adams and Kindred's work on state surfaces of alternating knots \cite{ak:alternating}.  They show that the $3$-dimensional geography of an alternating knot is the union of the wedges of the non-orientable basic state surfaces of a reduced alternating diagram, together with the wedges of the orientable basic state surfaces with twisted bands added.  Since the only basic state surfaces for the diagram of the trefoil in Figure~\ref{fig:trefoil-mobius} are the Seifert surface and the M\"obius strip, we obtain the geography in Figure~\ref{fig:trefoil-geography-3}.

% *****
\subsection{Essential Surfaces}
\label{ssec:essential}

While the geography of alternating knots can be analyzed using state surfaces, not every class of knots has such a known ``generating set''. In order to understand these classes, we must consider the set of incompressible and $\partial$-incompressible surfaces in the complement of $K$.  

We briefly recall some definitions of the required tools. Suppose $Y$ is a $3$-manifold with boundary and $F$ is a properly embedded surface in $Y$. Below, we will only consider when $Y$ is the exterior of a knot $K$, i.e.\ when $Y$ is $E(K)= S^3 \setminus \nu(K)$. Choosing a meridian $\mu$ and Seifert-framed longitude $\lambda$ for $\partial \nu(K)$, we write $[\partial F] = a[\mu] + b[\lambda]$.  The \dfn{boundary slope} $s(F)$ is $\frac{a}{b}\in \qq \cup \{\frac{1}{0}\}$. This agrees with the definition of boundary slope for spanning surfaces, as for a spanning surface, $b=1$ and $a$ measures the linking number with respect to the Seifert framing.  As before, the \dfn{normal Euler number} $e_3(F)$ is the negation of the boundary slope.  We extend the Euler-normalized first Betti numbers $\Gamma^\pm$ to any surface properly embedded in the exterior of a knot.

We follow \cite{ht:2-bridge-incompressible} in defining a \dfn{compressing disk} for $F$ to be an embedded disk $D \subset Y$ with $D \cap F = \partial D$.  The surface $F$ is \dfn{incompressible} if for each compressing disk $D$, there is a disk $D' \subset F$ with $\partial D' = \partial D$.  Similarly, a \dfn{$\partial$-compressing disk} of $F$ is an embedded disk $D \subset Y$ with $D \cap F = \alpha$ and $D \cap \partial Y = \beta$, where $\alpha \cup \beta = \partial D$ and $\alpha \cap \beta = S^0$.  The surface $F$ is \dfn{$\partial$-incompressible} if for each $\partial$-compressing disk $D$, there is a disk $D' \subset F$ with $\partial D' = \alpha \cup \beta'$ with $\beta' \subset \partial F$.  A \dfn{compression} along a compressing or $\partial$-compressing disk is a surgery on $F$ along the compressing or $\partial$- compressing disk. Finally, we say that a surface that is  incompressible and  $\partial$-incompressible, or is a sphere not bounding a ball, or is a disk that is not $\partial$-parallel, is \dfn{essential}.

To understand how essential surfaces are related to the $3$-dimensional geography of a knot $K$, we define, for any $(u,v) \in \qq \times \nn$, the \dfn{rational wedge based at $(u,v)$} by
\[\tilde{W}_{(u,v)} = \left\{(e,b) \in \qq \times \nn \;:\; |e-u| \leq 2(b-v) \right\}.\]
If $F$ is a properly embedded surface in $E(K)$, we set the notation $\tilde{W}_F = \tilde{W}_{(e_3(F), b_1(F))}$.  Now let $\mathcal{F}_{\text{ess}}(K)$ denote the set of all incompressible and  $\partial$-incompressible surfaces in $E(K)$, which is a finite set by \cite{hatcher:finite-boundary-slopes}.  Define the \dfn{incompressible wedge} of $K$ to be
\[\tilde{W}(K) = \bigcup_{F \in \mathcal{F}_{\text{ess}}(K)} \tilde{W}_F.\]

We may now state the main structural result in this section, which constrains the non-orientable $3$-dimensional geography of a knot $K$ using knowledge of the incompressible and $\partial$-incompressible surfaces in its exterior. This should be thought of as an analogue of Equation~\eqref{eq:wedge-geography}. 

\begin{prop}
\label{prop:incompr-constraint}
	For each knot $K$, $R_3(K) \subset \tilde{W}(K)$.
\end{prop}

The key technical idea in the proof is that we can bound the change in boundary slopes under a $\partial$-compression.

\begin{lem}
\label{lem:slope-compression-bound}
	Given a knot $K$ and a properly embedded surface $F \subset E(K)$, let $F'$ be the result of a $\partial$-compression of $F$. The difference between the boundary slopes (and hence the normal Euler number) is bounded by $2$, i.e.\ 
	\[\left| e_3(F') - e_3(F) \right| \leq 2.\]
	If $F$ is incompressible, then we get equality in the bound above if and only if both $F$ and $F'$ are spanning surfaces.
\end{lem}

\begin{proof}
	We first suppose that $F$ has boundary slope $0$, i.e.\ $\partial F$ is a canonical longitude $\lambda$ of the boundary torus $T$. Recall that a $\partial$-compressing disk $D$ imprints an embedded arc $\beta$ on $T$ that intersects $\partial F$ precisely at its endpoints.  
	
	Arguing as in the proof of Theorem 1 of \cite{ht:2-bridge-crosscap}, we see that $\beta$ either intersects $\partial F$ on one or two sides in $T$.  If $\beta$ intersects $\partial F$ on one side, then $\beta$ and a subarc of $\partial F$ bound a disk in $T$.  That disk, when pushed slightly into $E(K)$, combines with the $\partial$-compressing disk $D$ to yield a compressing disk $D'$.  Thus,  $\partial$-compression along $D$ is tantamount to compression along $D'$, and hence $F$ and $F'$ have the same boundary slope.
	
\begin{figure}
\labellist
\small\hair 2pt
 \pinlabel {(a)} [t] at 62 11
 \pinlabel {(b)} [t] at 207 11
 \pinlabel {$\lambda$} [t] at 20 14
 \pinlabel {$\mu$} [r] at 5 29
 \pinlabel {$n$} [l] at 24 95
 \pinlabel {$\beta$} [b] at 100 100
 \pinlabel {$\partial F$} [b] at 100 56
\endlabellist
\begin{center}
\includegraphics{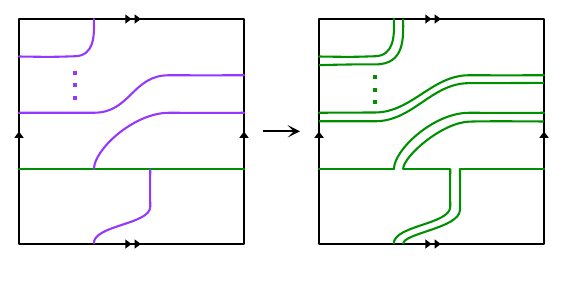}
\caption{(a) The curve $\beta$ when it intersects $\partial F$ on two sides. (b) The result at the boundary of a $\partial$-compression along $D$.}
\label{fig:beta-curve}
\end{center}
\end{figure}
	
	If, on the other hand, $\beta$ intersects $F$ on two sides, then up to isotopy, $\beta$ must be of the form pictured in Figure~\ref{fig:beta-curve}(a) for some $n \in \zz$. That is, after adding a subarc of $\partial F$ to obtain a closed curve $\hat{\beta}$, we must have that $[\hat{\beta}]$ is of the form $n \mu + \lambda$.  Performing a $\partial$-compression along $D$ yields $2n$ additional longitudes and $2$ additional meridians for $\partial F'$; see Figure~\ref{fig:beta-curve}(b). In particular, we have $[\partial F'] = 2\mu + (2n+1)\lambda$.  
	
%	To determine when we get the equality $|e(F')-e(F)|=2$ in the case that $F$ is incompressible and has boundary slope $0$, we note that $F$ is a spanning surface and that $F'$ is a spanning surface if and only if $n=0$, i.e.\ if and only if $|e(F')| = 2$.
	
	More generally, suppose that $F$ has boundary slope $\frac{p}{q}$.  Consider the automorphism $\phi$ of $T$ represented by the matrix \[\begin{bmatrix} r & p \\ s & q \end{bmatrix} \in SL(2,\zz).\]
In particular, we may assume that $rq-ps = 1$ and $0 \leq s < q$. The automorphism $\phi$ carries the longitude $\lambda$ to $\partial F$.  By the first case, we know that $\phi^{-1}(\partial F') = 2\mu + (2n+1)\lambda$ for some $n \in \zz$.  Thus, the boundary slope of $F'$ may be written as
\[s(F') = \frac{2r+(2n+1)p}{2s+(2n+1)q}.\]
We then use the fact that $\phi \in SL(2,\zz)$ to compute that
\begin{equation} \label{eq:slope-bound}
	|s(F') - s(F)| = \left|\frac{2}{q(2s+(2n+1)q)}\right| \leq 2,
\end{equation}
which proves the bound stated in the lemma.

To prove the last statement of the lemma, first note that the incompressibility of $F$ implies, as above, that $\beta$ intersects $\partial F$ on two sides in $\partial E(K)$. Next,  if we have equality in Equation~\eqref{eq:slope-bound}, then $q=1$. When $q=1$, since $0 \leq s < q$, we see that $s$ must be $0$, and hence so must $n$.  Thus, we see that $F$ and $F'$ are both spanning surfaces.  Conversely, if $F$ is a spanning surface, then $q=1$ (and hence $s=0$), and if $F'$ is a spanning surface, then $1 = 2s+(2n+1)q = 2n+1$, and hence $n=0$, which yields equality in Equation~\eqref{eq:slope-bound}.
\end{proof}

\begin{proof}[Proof of Proposition~\ref{prop:incompr-constraint}]
	Suppose that $(e,b) \in R_3(K)$, i.e.\ that there is a spanning surface $F$ for $K$ with $e_3(F) = e$ and $b_1(F) = b$.  By a sequence of compressions and $\partial$-compressions 
	\[F = F_0 \to F_1 \to \cdots \to F_n = F',\]
we arrive at an imcompressible and $\partial$-incompressible surface $F'$ properly embedded in the exterior $E(K)$.  

We claim that $(e,b) \in \tilde{W}_{F'}$; the proposition will then follow. In fact, it suffices to show that the line segment $\ell_i$ joining $(e_3(F_i),b(F_i))$ to $(e_3(F_{i+1}),b(F_{i+1}))$ has slope at least $\frac{1}{2}$ in absolute value.  If $F_{i+1}$ is obtained from $F_i$ by a compression, then the boundary slopes of $F_i$ and $F_{i+1}$ agree, so the line segment $\ell_i$ has infinite slope.  If $F_{i+1}$ is obtained from $F_i$ by a $\partial$-compression, then $b_1(F_i) - b_1(F_{i+1}) = 1$ and Lemma~\ref{lem:slope-compression-bound} shows that $ |e_3(F_i)-e_3(F_{i+1})| \leq 2$. Thus, we have that the slope of $\ell_i$ is at least $\frac{1}{2}$ in absolute value.
\end{proof}

As a corollary of the proof, we obtain a bound on the values of $\Gamma^\pm$ of surfaces properly embedded in the exterior of $K$. 

\begin{cor} \label{cor:rational-bound}
	Given a knot $K$ and a properly embedded $F \subset E(K)$, we have
	\[\Gamma^\pm(F) \geq \min \{\Gamma^\pm(F')\;:\; F' \in \mathcal{F}_{\text{ess}}(K)\}.\]
\end{cor}

%% file: geography4.tex
%!TEX root = non-ori-3d-4d.tex

Our next step is to describe the framework necessary to investigate the non-orientable $4$-genus of a knot.  In parallel to the $3$-dimensional setting, we first review the normal Euler number and Allen's geography problem in Section~\ref{ssec:euler}.  We finish the section by describing lower bounds on the non-orientable $4$-genus in Sections~\ref{ssec:signature} and \ref{ssec:bound}, including a proof of Proposition~\ref{prop:oss-sg}. 

% *****
\subsection{Normal Euler Number and $4$-Dimensional Geography}
\label{ssec:euler}

As in the $3$-dimensional setting, non-orientable fillings carry a second homological invariant: the normal Euler number, which captures the topology of the normal bundle.  We present two equivalent definitions of the normal Euler number of a \textbf{cobordism} between two links.  A cobordism is a properly embedded, compact surface $F \subset [0,1] \times S^3$ with boundary $\partial F = K_0 \cup K_1$ such that $K_i \subset \{i\} \times S^3$.  A filling for a knot $K$ is a cobordism with $K_1 = K$ and $K_0 = \emptyset$.

First, following \cite{gl:signature}, let $F'$ be the result of pushing $F$ along a nonzero section of its normal bundle.  Let $\partial F' = K_0' \sqcup K'_1$, and define the \dfn{normal Euler number} of $F$ by
\begin{equation} \label{eq:e-lk-4d}
	e_4(F) = \lk(K_0,K_0') - \lk(K_1,K_1').
\end{equation}
In the case where the cobordism $F$ arises from attaching a flat band to a diagram $D_0$ to obtain $D_1$, the normal Euler number may be computed using the blackboard framing in the formula in \cite[Lemma 4.2]{oss:unoriented}:
\begin{equation} \label{eq:e-wr-4d}
	e_4(F) = \wri(D_0) - \wri(D_1).
\end{equation}

Second, following \cite{oss:unoriented}, let $F''$ be a small transverse push-off of $F$ so that the pushoffs at the ends $K''_i$ both realize the Seifert framings of $K_i$.  The normal Euler number $e_4(F)$ may  be computed by choosing compatible local orientations of $T_xF$ and of $T_xF''$ at each point $x \in F \cap F''$, then comparing the resulting orientation on $T_xF \oplus T_xF''$ to the ambient orientation on $[0,1] \times S^3$, and finally adding up these local contributions.   

It is straightforward to check using Equation~\eqref{eq:e-lk-4d} that the normal Euler number of a spanning surface $F$ that has been pushed into $B^4$ agrees, up to a sign, with normal Euler number of $F$ defined in Section~\ref{ssec:bs}.  More precisely, we have:

\begin{lem} \label{lem:normal-euler-3-4}
	If $F$ is a spanning surface of a knot $K$, and $F'$ is the result of pushing $F$ into $B^4$ relative to $K$, then
	\[e_4(F') = e_3(F) = - s(F).\]
\end{lem}

As a result of this lemma, we no longer distinguish between the $3$- and $4$-dimensional normal Euler numbers of a surface $F$, denoting both by $e(F)$.

The normal Euler number satisfies several useful properties:  it is always even, it vanishes on orientable surfaces, it is additive under concatenation of cobordisms, and, as shown in \cite{Massey:Euler}, it satisfies
\begin{equation}
	e(F) \equiv 2b_1(F) \mod 4.
\end{equation}

\begin{defn} [\cite{allen:geography}] \label{defn:4Dgeography} 
The \dfn{four-dimensional geography} of a knot $K$ is the subset $R_4(K) \subset 2\zz \times \nn$ of all pairs $(e(F), b_1(F))$ for $F \in \spanning_4(K)$.   
\end{defn} 

We may translate the addition of a twisted band from the $3$- to the $4$-dimensional setting.  The definition of a wedge for the four-dimensional geography is equivalent to that presented in Section \ref{ssec:bs}. 

\begin{ex}
That the three- and four-dimensional geographies of the trefoil are identical follows from our work from Section \ref{sec:geography3} and Proposition 3.1 in \cite{allen:geography}. The $5_2$ knot, however, provides an example where the $3$-dimensional geography is a strict subset of the $4$-dimensional geography. In Figure~\ref{fig:52-band}, we display a filling $F$ of the $5_2$ knot arising from a flat band move.  Using Equation~\eqref{eq:e-wr-4d}, we compute that $e(F)=2$ and $b_1(F)=1$.  Thus, we see that the point $(2,1) \in R_4(K)$.  On the other hand, using the algorithm of Adams and Kindred \cite{ak:alternating}, the point $(2,1)$ is not in $R_3(K)$.  
\end{ex}

\begin{figure}
\labellist
\small\hair 2pt
 \pinlabel {$\simeq$} [ ] at 126 34
\endlabellist
\begin{center}
\includegraphics{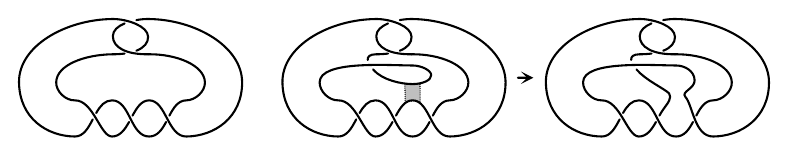}
\caption{A M\"obius strip filling of the $5_2$ knot with normal Euler number 2 arises from an isotopy, an attachment of a flat band, and finally an isotopy to the standard unknot.}
\label{fig:52-band}
\end{center}
\end{figure}

% *****
\subsection{Signature and the Gordon-Litherland Bound}
\label{ssec:signature}

To complement the constructions of non-orientable spanning surfaces and fillings discussed thus far, we now turn to lower bounds. We begin by recalling Gordon and Litherland's symmetric bilinear pairing $\langle,\rangle_F$ on the first homology of a compact embedded surface $F \subset S^3$.  The unit normal bundle $\nu_1(F)$ is a 2-fold cover $p: \nu_1(F) \to F$ and embeds into $S^3 \setminus F$.  For a pair of classes $a,b \in H_1(F)$ represented by embedded, oriented multicurves $\alpha, \beta \subset F$, define
\[\langle a,b \rangle_F = \lk (\alpha, p^{-1}\beta),\]
which yields a well-defined symmetric bilinear pairing \cite{gl:signature}. The signature $\sigma_{gl}$ of the pairing $\langle,\rangle_F$ is related to the signature of $K$ and the normal Euler number of $F$.

\begin{prop}[Corollary 5 of \cite{gl:signature}] \label{prop:gl-sign-sign-euler}
	If $F$ is a spanning surface of a knot $K$, then
	\[\sigma(K) = \sigma_{gl}(F) +\frac{e(F)}{2}.\]
\end{prop}

In particular, as noted in the introduction, we can rewrite the normalization in $\Gamma^\pm(F)$ as $\sigma_{gl}(F)$ rather than $\sigma(K) - \frac{e(F)}{2}$.

As a corollary, we obtain what is commonly termed the Gordon-Litherland inequality; see Corollary 2.5 of \cite{allen:geography}.

\begin{thm} \label{thm:GL}
Let $K \subset S^3$ be a knot and let $F$ be a filling of $K$.  Then
\[\left| \sigma(K) - \frac{e(F)}{2} \right| \leq b_1(F).\]
\end{thm}

Reinterpreting this inequality in terms of the Euler-normalized first Betti number, we obtain, for any filling $F$, the inequality
\[\Gamma^\pm_4(F) \geq 0,\]
hence justifying taking the minimum in the definition of the $\pm$-spanning surface defect.

The Gordon-Litherland inequality restricts the $3$- and $4$-dimensional geography of a knot $K$ to lie in the wedge $W_\sigma(K)$ based at $(2\sigma(K),0)$; see Figure~\ref{fig:bound-wedges}. As noted in the introduction, the spanning surface defects measure the minimum vertical distance from $R_*(K)$ to the boundary of $W_\sigma(K)$.

\begin{figure}
\labellist
\small\hair 2pt
 \pinlabel {$e$} [l] at 193 19
 \pinlabel {$b$} [b] at 171 94
 \pinlabel {$4\upsilon$} [t] at 99 14
 \pinlabel {$2\sigma$} [t] at 62 14
 \pinlabel {$\sigma-2\upsilon$} [b] at 117 74
 \pinlabel {$2\upsilon-\sigma$} [b] at 54 56
\endlabellist
\begin{center}
\includegraphics{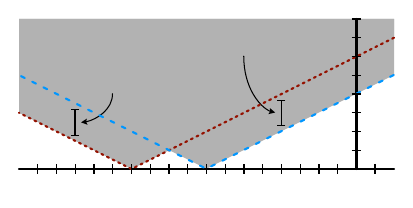}
\caption{The wedges $W_\sigma(K)$ and $W_\upsilon(K)$  restrict the $4$-dimensional (and hence $3$-dimensional) non-orientable geography of a knot $K$. The bound in Corollary~\ref{cor:oss} is the height of the vertex of $W_\sigma(K) \cap W_\upsilon(K)$, while the vertical distances between the edges of the wedges yield the bounds in Proposition~\ref{prop:oss-sg}.  The values for $\sigma$ and $\upsilon$ are taken from $T(4,3)$.}
\label{fig:bound-wedges}
\end{center}
\end{figure}

% *****
\subsection{Further Bounds}
\label{ssec:bound}

The Gordon-Litherland inequality is not the only bound on $b_1(F)$ and $e(F)$.  In fact, there is a class of what Sato \cite{sato:unori-slice-torus} calls ``non-orientable slice-torus invariants'' that yield structurally similar bounds.  Of importance in this paper is the $\upsilon$ invariant derived from Heegaard-Floer homology.

\begin{thm}[Theorem 1.1 of \cite{oss:unoriented}]
\label{thm:oss}
	For any filling $F$ of a knot $K$, we have
	\[\left| 2\upsilon(K) - \frac{e(F)}{2}\right| \leq b_1(F).\]
\end{thm}

As with the Gordon-Litherland inequality, Theorem~\ref{thm:oss} restricts the possible pairs $(e,b)$ for surfaces in $B^4$ spanning a given knot $K$ to a wedge $W_\upsilon(K)$ based at $(4\upsilon(K),0)$; see Figure~\ref{fig:bound-wedges} again.

The standard next step is to combine Theorems~\ref{thm:GL} and \ref{thm:oss} to obtain a lower bound on the non-orientable $4$-genus.  

\begin{cor}[Theorem 1.2 of \cite{oss:unoriented}]
\label{cor:oss}
	For any knot $K$, we have
	\[\left| \upsilon(K) - \frac{\sigma(K)}{2}\right| \leq \gamma_4(K).\]
\end{cor}

We are now ready to prove Proposition~\ref{prop:oss-sg}. The bound in Corollary~\ref{cor:oss} can be visualized as the height of the base of the wedge $W_\sigma(K) \cap W_\upsilon(K)$; see Figure~\ref{fig:bound-wedges} once again.  Inspired by this visualization, we may also extract a lower bound on the Euler-normalized non-orientable $4$-genus, as the quantity $| 2\upsilon(K) - \sigma(K)|$ measures the vertical distance between the boundaries of $W_\sigma(K)$ and  $W_\upsilon(K)$.  The proposition follows.

%% file: torus.tex
%!TEX root = non-ori-3d-4d.tex

Before tackling the main results of the paper, we take a motivational interlude to briefly discuss torus knots. Theorem~\ref{thm:alternating} implies that $3$- and $4$-dimensional spanning surface defects agree for alternating knots. The torus knot $T(4,3)$ displays similar behavior, even though its ordinary $3$- and $4$-dimensional genera differ.  Further, as with alternating knots, the bound in Theorem~\ref{thm:turaev} for $T(4,3)$ is sharp, though non-zero.

\begin{prop} \label{prop:torus}
	The torus knot $T4,3)$ has the following properties:
	\begin{enumerate}
	\item $\gamma_3(T(4,3)) > \gamma_4(T(4,3))$,
	\item $\sg^+_3(T(4,3)) = 0 = \sg^+_4(T(4,3)$ while $\sg^-_3(T(4,3)) = 2 = \sg^-_4(T(4,3)$, and
	\item $g_T(T(4,3)) = 1 = \asg_3(T(4,3))$.
	\end{enumerate}
\end{prop}

\begin{proof}
	The first part of the proposition follows from the computation that $\gamma_3(T(4,3)) = 2$ and $\gamma_4(T(4,3))=1$; see \cite[Example 5.3]{jvc:3vs4} or the entry for $8_{19} = T(4,3)$ in \cite{knotinfo}.  These values are realized using the $3$- and $4$-dimensional pinch surfaces $F_*(4,3)$ defined in \cite{jvc:3vs4} and \cite{batson:non-ori-slice}, respectively.  Using \cite[Lemma 2.3]{non-ori-torus}, we may compute that 
	\begin{align*}
		e(F_3(4,3)) &= -12 & e(F_4(4,3)) &= -10 \\
		b_1(F_3(4,3)) &= 2 & b_1(F_4(4,3)) &= 1.
	\end{align*}
	
	The proof of the second part starts with the computations $\sigma(T(4,3)) = -6$ and $\upsilon(T(4,3)) = -2$ using techniques from \cite{jvc:non-ori-milnor}. Thus, the bounds in Proposition~\ref{prop:oss-sg} are sharp at the pinch surface $F_4(4,3)$, and hence $\sg^+_4(T(4,3)) = 0$ and $\sg^-_4(T(4,3)) = 2$.  The computations for the second part of the proof are summarized in terms of geography in Figure~\ref{fig:t-4-3-geography}.
	
	Turning to spanning surfaces of $T(4,3)$ in $S^3$, on the positive side, the minimal Seifert surface $S$ has $e(S) = 0$ and $b_1(S) = 6$, and hence $\Gamma_3^+(\twist_+(S)) = 0$.  It follows that $\sg^+_3(T(4,3)) = 0$. On the negative side, the pinch surface $F_3(4,3)$ has first Betti number $2$ and normal Euler number $-12$, and hence has $\Gamma^-_3(F_3(4,3)) = 2$.  As the lower bounds in Proposition~\ref{prop:oss-sg} also apply to $\sg^\pm_3$, it follows that $\sg^-_3(T(4,3)) = 2$ as well.  Once again, see Figure~\ref{fig:t-4-3-geography}.
	
\begin{figure}
\labellist
\small\hair 2pt
 \pinlabel {$R_4(T(4,3))$} [t] at 93 5
 \pinlabel {$R_3(T(4,3))$} [t] at 277 5
 \pinlabel {$-10$} [t] at 40 17
 \pinlabel {$-12$} [t] at 207 17
 \pinlabel {$2$} [t] at 329 17
 \pinlabel {$1$} [r] at 124 30
 \pinlabel {$2$} [r] at 310 39
 \pinlabel {$6$} [r] at 310 83
 \pinlabel {$F_3$} [tr] at 205 40
 \pinlabel {$\twist_+(S)$} [tl] at 334 83
 \pinlabel {$F_4$} [r] at 33 30
 \pinlabel {$e$} [l] at 174 21
 \pinlabel {$b$} [b] at 127 111
 \pinlabel {$e$} [l] at 357 21
 \pinlabel {$b$} [b] at 312 111
\endlabellist
\begin{center}
\includegraphics{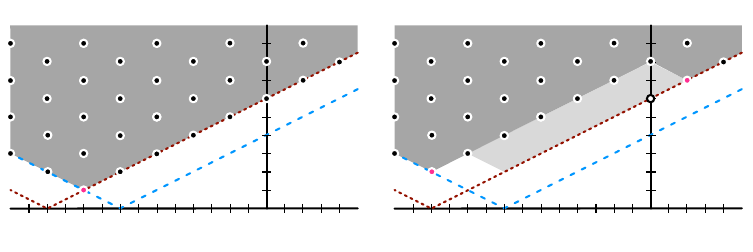}
\caption{The $4$-dimensional geography of the torus knot $T(4,3)$ is precisely $W_{F_4(p,q)}$, while the $3$-dimensional geography contains both $W_{F_3(p,q)}$ and $W_{\twist_+(S)}$ (and may contain more).}
\label{fig:t-4-3-geography}
\end{center}
\end{figure}

	Finally, we read off the computation $g_T(T(4,3)) = 1$ from \cite{ak:de-alt} or \cite{lowrance:kh-width}, while the work above yields $\asg_3(T(4,3)) = 1$.
\end{proof}

The proof of the first part of the proposition shows that $T(4,3)$ has a gap of $1$ between its non-orientable $3$- and $4$-genera.  In fact, Jabuka and Van Cott \cite{jvc:3vs4} prove that the gap can be arbitrarily large on torus knots:
\[(\gamma_3(T(qk+1,q)) - \gamma_4(T(qk+1,q)) \geq \frac{k}{2}.\] 

The torus knot $T(4,3)$ is also the first non-trivial example in the family of torus knots $T(2k,2k-1)$.  Batson \cite{batson:non-ori-slice} showed that this family has unbounded $\gamma_4$, with $\gamma_4(T(2k,2k-1))=k-1$.  Further, all elements in the family have $R_4 = W_\sigma \cap W_\upsilon$ \cite{non-ori-torus}, which yields
\[\sg^-_4(T(2k,2k-1))=0 \quad \text{and} \quad \sg^+_4(T(2k,2k-1))=2k-2.\]
%Thus, we obtain the result in Corollary~\ref{cor:torus-turaev}:
%\begin{align*}
%	g_T(T(2k,2k-1)) &\geq \asg_3(T(2k,2k-1)) \\
%	&\geq \asg_4(T(2k,2k-1)) \\
%	&= k-1.
%\end{align*}

We note that for $T(2k,2k-1)$, it is not clear if $\sg^+_3 = \sg^+_4$; in fact, this seems unlikely to be the case, as the Seifert surface lies inside the wedge of the $3$-dimensional pinch surface for $k>1$.

%% file: gap.tex
%!TEX root = non-ori-3d-4d.tex

In this section, we provide a new family of examples for which the gap between the non-orientable $3$- and $4$-genera becomes arbitrarily large.  These examples, which are a family of pretzel knots, have several interesting features. First,  the new examples all bound M\"obius bands in $B^4$ even though the orientable $4$-genus becomes arbitrarily large.  Second, the gap between the Euler-normalized non-orientable $3$- and $4$-genera is bounded above by $2$, demonstrating that the spanning surface defect gaps are independent of ordinary non-orientable gaps.

The main result of this section is Proposition~\ref{prop:pretzel-gap}, which asserts that for any $n \geq 1$ and positive odd integers $k,r, p_1, \ldots, p_{n-1}$, the pretzel knot $P = P(-k,r,-r-1, p_1, -p_1, \ldots,p_{n-1}, -p_{n-1})$ satisfies
\begin{enumerate}
	\item $\gamma_4(P) = 1$,
	\item $\gamma_3(P) = 2n$, and 
	\item $g_4(P) \geq \frac{1}{2}(k - r - 2)$.
\end{enumerate}
Such a knot is illustrated in Figure~\ref{fig:candidateknot} for $p_i = 3$.

\begin{figure}
\labellist
\small\hair 2pt
 \pinlabel {$-k$} [r] at 13 58
 \pinlabel {$r$} [r] at 55 58
 \pinlabel {$-r-1$} [l] at 106 66
 \pinlabel {$p_{n-1}$} [r] at 153 50
 \pinlabel {$-p_{n-1}$} [l] at 206 58
\endlabellist

	\begin{center}
	\includegraphics{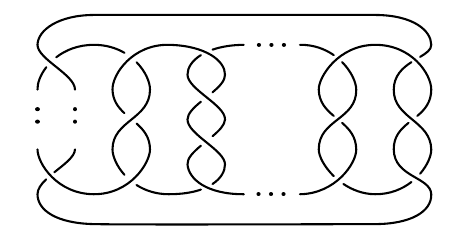}
	\end{center}
	\caption{The pretzel knot in Proposition~\ref{prop:pretzel-gap} with $r=3=p_i$.}
	\label{fig:candidateknot}
\end{figure}

\begin{proof}[Proof of Proposition~\ref{prop:pretzel-gap}]
	The fact that $\gamma_3(P) = 2n$ follows directly from the work of Ichihara and Mizushima \cite[Theorem 1.2]{im:pretzel}. 
	
To prove that $\gamma_4(P) = 1$, we begin by illustrating two types of band moves. A band move of Type I works for a pretzel of the form $P(q_1,...,q_{N-2},s,-s)$ for $s \in \zz$ ($s \neq 0$). We can perform this band move between the twists $q_{N-2}$ and $s$ to obtain the pretzel $P(q_1,...,q_{N-2})$ and an unlinked unknot; see Figure~\ref{fig:oddoddband}. These band moves are always orientable, and show that $P(q_1,...,q_{N-2})$ and $P(q_1,...,q_{N-2},s,-s)$ are concordant. 

\begin{figure}
\labellist
\small\hair 2pt
 \pinlabel {$\simeq$} [ ] at 221 64
\endlabellist
\begin{center}
	\includegraphics{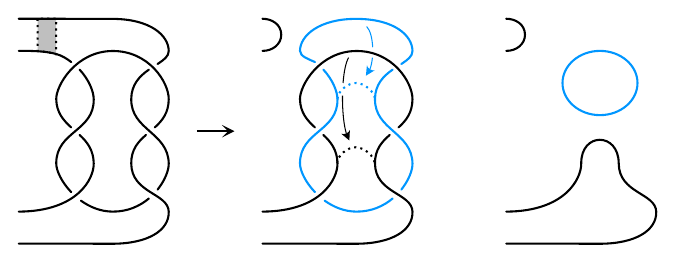}
	\caption{A band move of Type I with $s = 3$.}
	\label{fig:oddoddband}
\end{center}
\end{figure}

Next, consider a pretzel of the form $P(q_1,...,q_{N-2},s,-(s+1))$ for $s \in \zz$ ($s \neq 0, -1$).  A band move of Type II between $q_{N-2}$ and $s$, pictured in Figure~\ref{fig:evenoddband}, reduces the knot to $P(q_1,...,q_{N-2})$. Note that this band move may be non-orientable. 

\begin{figure}
\labellist
\small\hair 2pt
 \pinlabel {$\simeq$} [ ] at 221 64
\endlabellist
\begin{center}
	\includegraphics{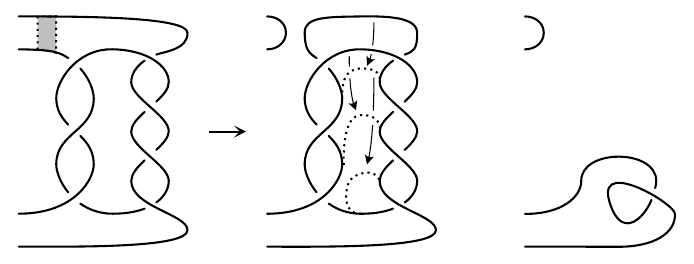}
	\caption{A band move of Type II with $s = 3$.}
	\label{fig:evenoddband}
\end{center}
\end{figure}

	We are now ready to prove that $\gamma_4(P) = 1$.  Applying Type I moves $n-1$ times, we see that $P$ is concordant to $P(-k, r, -r-1)$.  A further band move of Type II (which is non-orientable) transforms $P(-k, r, -r-1)$ into an unknot, and hence shows that $P$ is filled by a M\"obius band.  Note that a diagrammatic computation using Equation~\eqref{eq:e-wr-4d} implies that the normal Euler number of the surface $F$ constructed in this proof is 
\begin{equation} \label{eq:pretzel-e}
e(F) = 2(k-r-1).
\end{equation}

	It remains to be shown that $g_4(P) \geq \frac{1}{2}(k-r-2)$, which is achieved by calculating the signature of $P$.  Since $P$ is concordant to $P(-k, r, -r-1)$, it suffices to compute the signature of the latter knot, which can be achieved either directly using the Goeritz matrix as in \cite{gl:signature}, or using the formulae for signatures of pretzel knots in \cite{jabuka:witt} or \cite{shinohara:pretzel-signature} (correcting for sign conventions). The result is that:
	
\begin{equation} \label{eq:sgn-pretzel}
	\sigma(P) = \begin{cases}
		k-r & k < r(r+1), \\
		k-r-2 & k > r(r+1).
	\end{cases}
\end{equation}

	The lower bound for $g_4(P)$ now follows from  applying Murasugi's signature bound \cite{murasugi:signature}.
\end{proof}

As mentioned at the beginning of this section, despite the large gaps between $\gamma_3$ and $\gamma_4$ for the knots in Proposition~\ref{prop:pretzel-gap}, these examples do not exhibit similarly large gaps between $\sg^\pm_3$ and $\sg^\pm_4$. 

\begin{prop} \label{prop:pretzel-no-gap}
	The pretzel knots $P$ in Proposition~\ref{prop:pretzel-gap} satisfy
	\[\sg^\pm_3 (P) - \sg^\pm_4(P) \leq 2.\]
\end{prop}

We conjecture that $\sg^\pm_3 (P) - \sg^\pm_4(P) = 0$.  The conjecture would hold if the M\"obius band filling constructed in the proof of Proposition~\ref{prop:pretzel-gap} were the unique such band.

\begin{proof}
	Let $F_A$ (resp.\ $F_B$) be the surfaces for $P$ depicted in Figure~\ref{fig:pretzel-state}. Direct computations on the diagrams yield
\begin{equation*}
	e(F_A) = 2k+\sum_{i=1}^{n-1} p_i \quad \text{and} \quad
	e(F_B) = -4r-2-\sum_{i=1}^{n-1} p_i,
\end{equation*}
as well as 
\begin{equation*}
	b_1(F_A) = 2+r + \sum_{i=1}^{n-1} p_i \quad \text{and} \quad
	b_1(F_B) = 1+r+k + \sum_{i=1}^{n-1} p_i.
\end{equation*}

\begin{figure}
\begin{center}
	\includegraphics{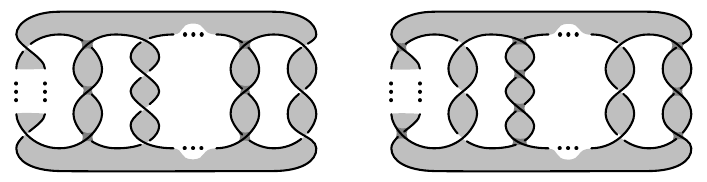}
\end{center}
\caption{The surfaces $F_A$ (left) and $F_B$ (right) used in the proof of Proposition~\ref{prop:pretzel-no-gap}.}
\label{fig:pretzel-state}
\end{figure}

Putting these computations together with the formula for the signature of $P$ in Equation~\eqref{eq:sgn-pretzel} yields
\begin{align*}
	\Gamma^-(F_A) &= \begin{cases} 2 & k < r(r+1) \\ 0 & k > r(r+1) \end{cases}\\
	\Gamma^+(F_B) &= \begin{cases} 0 & k < r(r+1) \\ 2 & k > r(r+1) \end{cases}
\end{align*}
The proposition follows.
\end{proof}

%% file: stable-gap-pretzel.tex
%!TEX root = non-ori-3d-4d.tex

In this section, we prove Proposition~\ref{prop:pretzelside} by exhibiting an infinite family of non-alternating slice pretzel knots with negative Euler-normalized non-orientable $3$-genus $2$.  The techniques deployed in the proof are interesting more for their suggestion of what work would be necessary for the computation of the spanning surface defect of adequate knots and the proof of Conjecture~\ref{conj:adequate-turaev} than for the result itself which, as noted above, would also follow from Theorems~\ref{thm:alternating} and \ref{thm:turaev}. The idea is reminiscent of the computation of the $3$-dimensional geography of the trefoil, but instead of using Adams and Kindred's idea that basic state surfaces ``generate'' the geography of alternating knots, we use instead the constraint on the geography coming from essential surfaces described in Proposition~\ref{prop:incompr-constraint}.  That is, for pretzel knots, essential surfaces play the role that basic state surfaces play for alternating knots.  This means that we need to understand all possible essential surfaces in the complement of a pretzel knot.

% *****
\subsection{Essential Surfaces and Edgepath Systems}
\label{ssec:ho}

As a first step in the proof of Proposition~\ref{prop:pretzelside}, we describe the work of Hatcher and Oertel \cite{ho} on essential surfaces in the exteriors of Montesinos knots, specializing our description to the pretzel knots $P_n = P(-3,3,n)$, with $n \geq 3$ odd, as in \cite{im:montesinos-slope-bound} or \cite{lvdv:pretzel-slope}, when appropriate.

The central idea is that every essential surface belongs to a set $\candidate$ of candidate surfaces that is constructed as follows. Suppose that $K$ is a Montesinos knot composed of $N$ rational tangles. Divide $S^3$ into $N$ copies of $B^3$ that all meet along a common $S^1$ axis, with the $k^{th}$ $3$-ball containing the $k^{th}$ tangle; see Figure~\ref{fig:divide-S3}.  We produce candidate surfaces in $\candidate$ by constructing subsurfaces that lie in the $3$-balls and then ensuring that those subsurfaces match up along the common boundaries.  The subsurfaces themselves are described by a sequence of saddles that is related to the twists that make up the rational tangles.  

\begin{figure}
\labellist
\small\hair 2pt
 \pinlabel {$B_2$} [ ] at 99 90
 \pinlabel {$B_3$} [ ] at 112 132
 \pinlabel {$B_1$} [ ] at 44 22
\endlabellist
\begin{center}
\includegraphics{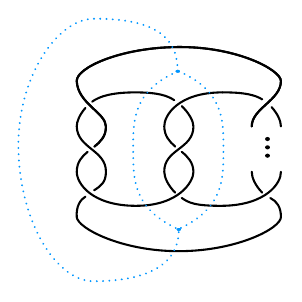}
\caption{Divide $S^3$ into $N$ copies of $B^3$, with each tangle of $K$ contained in its own $3$-ball $B_i$.  The two dots where the three $B_i$ meet represent the points where the common $S^1$ axis meets the projection $2$-sphere.}
\label{fig:divide-S3}
\end{center}
\end{figure}

One of the insights of Hatcher and Oertel is that these subsurfaces can be described and analyzed using combinatorial data.  The first step in the combinatorial description of a subsurface in a tangle is to record it as a path in the edgepath diagram $\mathcal{D} \subset \rr^2$, with coordinates on $\rr^2$ denoted $(u,v)$.  The diagram $\mathcal{D}$ consists of a collection of line segments with disjoint interiors that connect vertices of the form $\langle p/q\rangle = (1-1/q, p/q)$,  $\langle p/q\rangle^{\circ} = (1, p/q)$, and $\langle 1/0 \rangle = (-1,0)$  for $q>1$, $p \in \zz$ and $p$ and $q$ relatively prime.  There is a line segment joining $\langle p/q\rangle$ and $\langle r/s\rangle$ exactly when $ps - rq = \pm 1$; we denote such a line segment by $\langle p/q \rangle - \langle r/s \rangle$.  The diagram $\mathcal{D}$ also contains the horizontal segments  $\langle p/q \rangle - \langle p/q \rangle^{\circ}$, and all segments $\langle 1/0 \rangle - \langle p/1 \rangle$; see Figure~\ref{fig:edgepath-diagram}.

\begin{figure}
\labellist
\small\hair 2pt
 \pinlabel {$u$} [l] at 122 109
 \pinlabel {$v$} [v] at 63 213
\footnotesize\hair 2pt
 \pinlabel {$\langle \frac{1}{0} \rangle$} [br] at 21 112
 \pinlabel {$\langle 1 \rangle$} [br] at 62 152
 \pinlabel {$\langle 2 \rangle$} [br] at 62 193
 \pinlabel {$\langle 0 \rangle$} [br] at 62 112
 \pinlabel {$\langle -1 \rangle$} [tr] at 62 67
 \pinlabel {$\langle -2 \rangle$} [tr] at 62 23
 \pinlabel {$\langle 0 \rangle$} [tr] at 166 43
 \pinlabel {$\langle 1 \rangle$} [br] at 166 176
 \pinlabel {$\langle 0 \rangle^\circ$} [tl] at 299 43
 \pinlabel {$\langle 1 \rangle^\circ$} [bl] at 299 176
 \pinlabel {$\langle \frac{1}{2} \rangle^\circ$} [l] at 299 109
 \pinlabel {$\langle \frac{1}{3} \rangle^\circ$} [l] at 299 88
 \pinlabel {$\langle \frac{1}{2} \rangle$} [r] at 227 109
 \pinlabel {$\langle \frac{1}{3} \rangle$} [r] at 248 89
% \pinlabel {$q$} [ ] at 256 59
 \pinlabel {$\langle \frac{1}{5} \rangle$} [t] at 270 66
\endlabellist
\begin{center}
\includegraphics{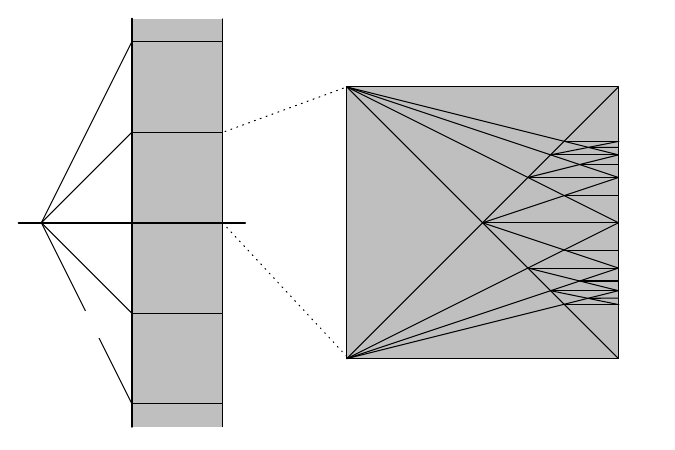}
\caption{The edgepath diagram $\mathcal{D}$ with the part of the diagram in $[0,1] \times [0,1]$ expanded.}
\label{fig:edgepath-diagram}
\end{center}
\end{figure}

A \dfn{basic edgepath} $\bep$ describes a sequence of saddles in a tangle which must satisfy the following properties:

\begin{enumerate}
\item The path $\bep$ begins on the edge $\langle p/q\rangle-\langle p/q \rangle^{\circ}$, and is constant if the starting point is not $\langle p/q \rangle$.

\item The path $\bep$ proceeds monotonically from right to left, with movement along vertical edges permitted.

\item The path $\bep$ neither retraces itself nor travels along two edges of a triangle in $\mathcal{D}$ in succession.  
\end{enumerate}

As noted in \cite{im:pretzel} or \cite{lvdv:pretzel-slope}, it is straightforward to see that, for a pretzel knot, a basic edgepath takes one of the following two forms:
\begin{enumerate}
\item $\sigma_{\pm p} = \langle 0 \rangle -\langle \pm 1/p\rangle$
\item $ \rho_{\pm p} =  \langle \pm 1 \rangle -\langle \pm 1/2 \rangle - \dots - \langle \pm 1/(|p|-1) \rangle - \langle \pm 1/|p|\rangle$
\end{enumerate}
See Figure~\ref{fig:basic-edgepath-surfaces} for illustrations of surfaces corresponding to several basic edgepaths and their extensions to include the point $\langle 1/0 \rangle$. Note that an edgepath does not necessarily determine a unique subsurface, but it does uniquely determine the first Betti number and normal Euler number of a subsurface.

\begin{figure}
\labellist
\small\hair 2pt
 \pinlabel {$\sigma_3$} [c] at 58 6
 \pinlabel {$\langle 1/0 \rangle - \sigma_3$} [c] at 181 6
 \pinlabel {$\langle 1/0 \rangle - \rho_3$} [c] at 303 6
\endlabellist
\begin{center}
\includegraphics{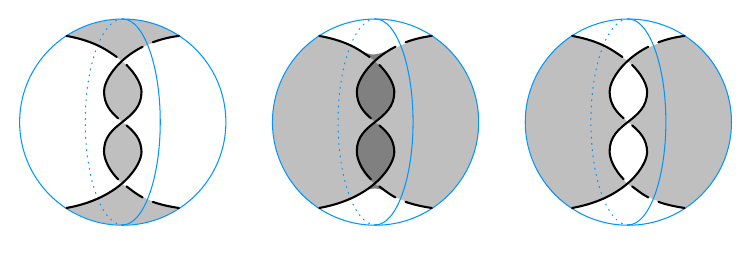}
\caption{Examples of subsurfaces in a tangle corresponding to the edgepaths $\sigma_3$, $\langle 1/0 \rangle - \sigma_3$, and $\langle 1/0 \rangle - \rho_3$.  The longitudes represent the common axis of the $3$-balls.}
\label{fig:basic-edgepath-surfaces}
\end{center}
\end{figure}

A \dfn{basic edgepath system} is simply a collection of basic edgepaths, one for each tangle.  There are eight possible basic edgepath systems for the pretzel knot $P_n$:
\begin{align*}
\bep_0&: \{\sigma_{-3}, \sigma_3, \sigma_n\} & \bep_4&: \{\rho_{-3}, \sigma_3, \sigma_n\} \\
\bep_1&: \{\sigma_{-3}, \sigma_3, \rho_n\} & \bep_5&: \{\rho_{-3}, \sigma_3, \rho_n\} \\
\bep_2&: \{\sigma_{-3}, \rho_3, \sigma_n\} & \bep_6&: \{\rho_{-3}, \rho_3, \sigma_n\} \\
\bep_3&: \{\sigma_{-3}, \rho_3, \rho_n\} & 
\bep_7&: \{\rho_{-3}, \rho_3, \rho_n \}
\end{align*}

In order to ensure that the edgepaths correspond to subsurfaces that match along their boundaries, we  require that their (left) endpoints have the same $u$ coordinate, and at that $u$ value, the $v$ coordinates sum to zero.  In order to achieve this condition, we must extend or truncate basic edgepath systems in one of the following three ways:

\begin{description}
\item[Type I] Type I edgepath systems are formed by first extending basic edgepaths by appending an edge $\langle 1/p\rangle-\langle 1/p\rangle^{\circ}$. The extended system, which we shall continue to call $\bep = \{\bep_1, \ldots, \bep_k\}$, can then be considered as a function $[0,1] \to \rr$ defined by $\bep(u) = \sum^k_{i=1} \bep_i(u)$.  Solve the equation $\bep(u) = 0$.  For each solution $u_0$, a new Type I edgepath system $\ep$ with edgepaths $\ep_i$ is constructed according to the following algorithm: if $u_0 \leq (|p_i|-1)/|p_i|$, then let \[\ep_i = \bep_i \cap \{(u,v): u \geq u_0 \}.\]
This procedure may cut an edge in its interior.  Denote such a partial edge on $\langle p/q \rangle - \langle r/s \rangle$ by
\[\left(\frac{k}{k+l}\langle p/q \rangle - \frac{l}{k+l} \langle r/s \rangle \right) - \langle r/s \rangle,\]
where $k+l$ is the number of sheets of the surface, and
\[u_0 = \frac{k(q-1)+ l(s-1)}{kq+ls}.\]
See \cite[p.\ 455]{ho}.

\item[Type II] Add vertical edges to edgepaths in $\bep$ so that the $v$-coordinates of the endpoints sum to 0 when $u=0$, while still adhering to the rule that an edgepath must not trace two sides of the same triangle. Only the minimal number of vertical edges need be added, as any additional pairs of edges will raise the first Betti number without changing the boundary slope.

\item[Type III] Complete each basic edgepath with a segment connecting the left endpoint to $\langle 1/0 \rangle$.
\end{description}
The result of any one of these operations is an \dfn{edgepath system} $\ep$.

We are now in a position to enumerate all possible edgepath systems that describe surfaces in the candidate set $\candidate$ for the pretzel knot $P_n$.  According to the proof of Lemma 3.7 in \cite{im:pretzel}, there is a unique Type I edgepath system of the form 
\[\ep^I = \left\{ 
\begin{array}{l}
\left(\frac{4}{n+1} \left\langle -\frac{1}{2} \right\rangle + \frac{n-3}{n+1} \left\langle -\frac{1}{3} \right\rangle\right)-\left\langle-\frac{1}{3}\right\rangle, \\
\left(\frac{2}{n+1}\left\langle0\right\rangle+\frac{n-1}{n+1}\left\langle\frac{1}{3}\right\rangle\right)-\left\langle\frac{1}{3}\right\rangle, \\
\left(\frac{n-1}{n+1}\left\langle0\right\rangle+\frac{2}{n+1}\left\langle\frac{1}{n}\right\rangle\right)-\left\langle\frac{1}{n}\right\rangle
\end{array}
\right\}\]
See Figure~\ref{fig:edgepath-systems}(a) for an illustration of this system.

\begin{figure}
\labellist
\small\hair 2pt
 \pinlabel {$u$} [l] at 87 79
 \pinlabel {$v$} [b] at 16 150
 \pinlabel {$u$} [l] at 193 79
 \pinlabel {$v$} [b] at 120 150
 \pinlabel {$u$} [l] at 359 79
 \pinlabel {$v$} [b] at 286 150
 \pinlabel {$\frac{4}{7}$} [b] at 50 150
 \pinlabel {(a)} [t] at 44 5
 \pinlabel {(b)} [t] at 151 5
 \pinlabel {(c)} [t] at 286 5
\endlabellist

\begin{center}
\includegraphics{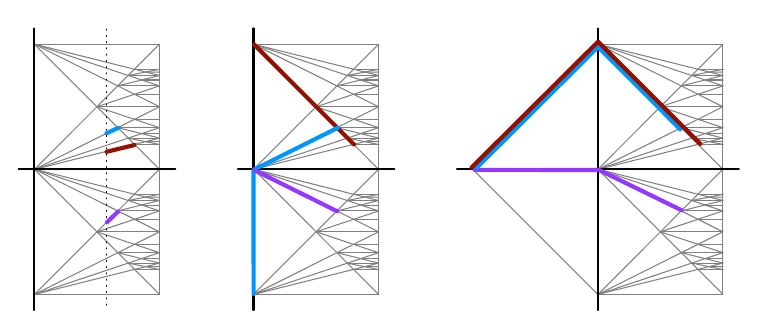}
\caption{(a) The edgepath system $\ep^I$, (b) the edgepath system $\ep^{II}_1$, and (c) the edgepath system $\ep^{III}_3$, all with $n=5$.}
\label{fig:edgepath-systems}
\end{center}
\end{figure}

To specify Type II edgepath systems, we use a superscript $+$ (resp.\ $-$) to denote the addition of an upward-pointing (resp.\ downward-pointing) vertical edge.  In some cases, there are choices about where to attach a vertical edge, but such choices will be immaterial to later computations of the first Betti number and the normal Euler number. Up to such choices, the Type II edgepath systems are as follows:

\begin{align*}
\ep^{II}_0&: \{\sigma_{-3}, \sigma_3, \sigma_n\} & \ep^{II}_4&: \{\rho_{-3}, \sigma_3^+, \sigma_n\} \\
\ep^{II}_1&: \{\sigma_{-3}, \sigma_3^-, \rho_n\} & \ep^{II}_5&: \{\rho_{-3}, \sigma_3, \rho_n\} \\
\ep^{II}_2&: \{\sigma_{-3}, \rho_3, \sigma_n^-\} & \ep^{II}_6&: \{\rho_{-3}, \rho_3, \sigma_n\} \\
\ep^{II}_3&: \{\sigma_{-3}^{--}, \rho_3, \rho_n\} & 
\ep^{II}_7&: \{\rho_{-3}^-, \rho_3, \rho_n \}
\end{align*}
The edgepath system $\ep^{II}_0$ corresponds to the Seifert surface of $P_n$.  See Figure~\ref{fig:edgepath-systems}(b) for an illustration of a different Type II system, and see Figure~\ref{fig:pretzel-surfaces}(a) for an illustration of the Seifert surface.

\begin{figure}
\labellist
\small\hair 2pt
 \pinlabel {(a)} [t] at 75 145
 \pinlabel {(b)} [t] at 227 145
 \pinlabel {(c)} [t] at 75 9
 \pinlabel {(d)} [t] at 227 9
\endlabellist
\begin{center}
\includegraphics{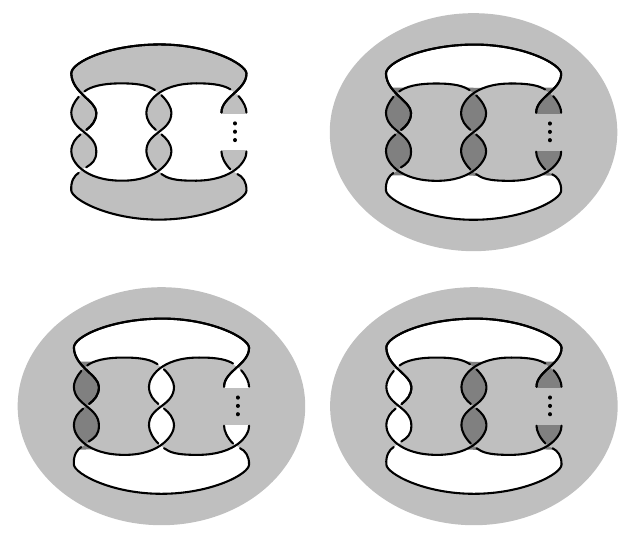}
\caption{Spanning surfaces corresponding to several edgepath systems: (a) the Seifert surface $\ep^{II}_0$, (b) the surface $\ep^{III}_0$, (c) the surface $\ep^{III}_3$, and (d) the surface $\ep^{III}_4$.}
\label{fig:pretzel-surfaces}
\end{center}
\end{figure}

Finally, Type III edgepath systems build directly off of the basic edgepath systems; we denote them by $\ep^{III}_k$ for $k =0, \ldots, 7$. See Figure~\ref{fig:edgepath-systems}(c) for one example of a Type III edgepath system, and see Figure~\ref{fig:pretzel-surfaces}(b,c,d) for illustrations of surfaces that correspond to Type III edgepath systems.

% *****
\subsection{Computations on Edgepath Systems}
\label{ssec:ho-comp}

Now that we know how to construct edgepath systems that describe candidate surfaces, we proceed to describe how to compute the first Betti number and the normal Euler number of those candidates from their edgepath systems.  We continue to follow \cite{ho, im:montesinos-slope-bound,lvdv:pretzel-slope}.

We first define several quantities related to an edge in an edgepath.  
\begin{itemize}
\item The \dfn{length} $|e|$ of an edge $e$ is $1$ for a complete edge between two vertices, and is $\frac{k}{k+l}$ for a partial edge $\left(\frac{k}{k+l}\langle p/q \rangle + \frac{l}{k+l}\langle r/s \rangle\right) - \langle r/s \rangle$. The length $| \ep|$ of an edgepath system is the sum of the lengths of its constituent edges that lie in the half-plane $u \geq 0$.

\item The \dfn{sign} $\varepsilon(e)$ of an edge $e$ is $+1$ if the the edge increases in height from right to left, or is an upward-pointing vertical edge, and is $-1$ if the edge decreases in height from right to left, or is a downward-pointing vertical edge.  The sign of the edge $e = \langle p/q \rangle - \langle r/s \rangle$ can also be computed as $\varepsilon(e) = ps-qr$.

\item The \dfn{twist} $\tau(e)$ of an edge $e$ is $\tau(e) = -2\varepsilon(e)|e|$ if $e$ lies in the right half-plane $u \geq 0$, and $\tau(e) = 0$ if $e$ connects an integer point $\langle n/1 \rangle$ to $\langle 1/0 \rangle$.  The twist $\tau(\ep)$ of an edgepath system is the sum of the twists of its constituent edges; as with the length, this is equivalent to the sum of the twists of its constituent edges that lie in the half-plane $u \geq 0$.
\end{itemize}

The computation of the Euler characteristic of a surface corresponding to an edgepath system depends on its type, and further, on the number of sheets $\#\ep$ of the surface described by the edgepath system.  For Type II and Type III surfaces, we may assume for our purposes that the number of sheets is always one, but the Type I surfaces can have multiple sheets.
\begin{description}
\item[Type I] For a Type I edgepath system with $N$ subsurfaces, $\#\ep$ sheets, and left endpoint at $u_0$, we have
\[\frac{\chi(\ep)}{\#\ep} = -|\ep| + N - \frac{1}{1-u_0}(N-2).\]

\item[Type II] For a Type II edgepath system, we have
\[\chi(\ep) = 2-|\ep|.\]

\item[Type III] For a Type III edgepath system, we have
\[\chi(\ep) = -|\ep|.\]

\end{description}

For the Type I edgepath system, we may calculate that $u_0 = \frac{2n-2}{3n-1}$ and that the associated surface has $\frac{n+1}{2}$ sheets; the division by $2$ arises since the fractions in $\ep^I$ have even numerators and denominators since $n$ is odd, so the denominator in lowest terms is $\frac{n+1}{2}$. Further, the three edgepaths have lengths $\frac{4}{n+1}$, $\frac{2}{n+1}$, and $\frac{n-1}{n+1}$. The formula above then yields $\chi(\ep^I) = -\frac{n+1}{2}$, so $b_1(\ep^I) = \frac{n+3}{2}$. The calculations of the first Betti numbers of the Type II and Type III appear in the third columns of Table~\ref{tab:compute}(a) and (b), respectively.

\begin{table}
\renewcommand*\arraystretch{1.5}
\begin{tabular}{|c| c c | c c|} 
 \hline
 $\ep^{II}_k$ & $e$ & $b_1$ & $\Gamma^+$ & $\Gamma^-$ \\
 \hline
 (0) & $0$ & $2$ & $2$ & \fbox{$2$} \\
 (1) & $2n-2$ & $n+1$ & $2$ & $2n$ \\
 (2) & $4$ & $4$ & $2$ & $6$ \\
 (3) & $2n+2$ & $n+3$ & $2$ & $2n+4$ \\
 (4) & $-4$ & $4$ & $6$ & \fbox{$2$} \\
 (5) & $2n-6$ & $n+1$ & $4$ & $2n-2$ \\
 (6) & $0$ & $4$ & $4$ & $4$ \\
 (7) & $2n-2$ & $n+3$ & $4$ & $2n+2$ \\ \hline
\end{tabular}
\quad 
\begin{tabular}{|c| c c | c c|} 
 \hline
 $\ep^{III}_k$ & $e$ & $b_1$ & $\Gamma^+$ & $\Gamma^-$  \\
 \hline
  (0) & $0$ & $4$ & $4$ & $4$ \\
 (1) & $2n$ & $n+2$ & $2$ & $2n+2$ \\
 (2) & $6$ & $5$ & $2$ & $8$ \\
 (3) & $2n+6$ & $n+3$ & \fbox{$0$} & $2n+6$\\
 (4) & $-6$ & $5$ & $8$ & \fbox{$2$} \\
 (5) & $2n-6$ & $n+3$ & $6$ & $2n$ \\
 (6) & $0$ & $6$ & $6$ & $6$ \\
 (7) & $2n$ & $n+4$ & $4$ & $2n+4$ \\ \hline
\end{tabular}
\vspace{.1in}

\caption{First Betti numbers, normal Euler numbers, and $\Gamma$ computations for (a) Type II edgepath systems and (b) Type III edgepath systems. Minimal values of $\Gamma^\pm$ are \fbox{boxed}.}
\label{tab:compute}
\end{table}

The computation for the boundary slope is similar, though simpler.  Denote the edgepath system corresponding to the Seifert surface of $P_n$ by $\ep_S$.  For a surface corresponding to the edgepath $\ep$, we have
\[s(\ep) = \tau(\ep) - \tau(\ep_S),\]
and hence
\[e(\ep) = \tau(\ep_S) - \tau(\ep).\]

As noted above, the Seifert surface corresponds to $\ep^{II}_0$, which has twist $2$.  For the Type I edgepath system, all three edgepaths have negative signs, and hence we obtain $e(\ep^I) = -\frac{8}{n+1}$. See Table~\ref{tab:compute} once again for the normal Euler numbers of the Type II and Type III surfaces.

We now have enough information for the proof of Proposition~\ref{prop:pretzelside}.

\begin{proof}[Proof of Proposition~\ref{prop:pretzelside}] The point of Hatcher and Oertel's algorithm is that every essential surface corresponds to one of the Type I, II, or III edgepath systems.

The calculations above show that the Type I surface has $\Gamma^-(\ep^I) = \frac{n+2}{2}-\frac{4}{n+1}$.  For odd $n \geq 3$, we have $\Gamma^-(\ep^I) \geq \frac{3}{2}$. Thus, any spanning surface that boundary compresses to the surface corresponding to $\ep^I$ must have $\Gamma^-(F) \geq 2$.  Further, Table~\ref{tab:compute} shows that $\Gamma^-$ is at least $2$ for each of the Type II and Type III surfaces.  Thus, Corollary~\ref{cor:rational-bound} shows that $\sg_3^-(P_n) \geq 2$. On the other hand, the surface corresponding to $\ep^{III}_3$ yields $\sg_3^+(P_n) = 0$.
\end{proof}

\begin{ex}
The information generated for the proof of Proposition~\ref{prop:pretzelside} also yields, via Proposition~\ref{prop:incompr-constraint}, the geography of the knots $P_n$.  See Figure~\ref{fig:P333-geography3} for the geography of $P_3$. 
\end{ex}

\begin{figure}
\labellist
\small\hair 2pt
 \pinlabel {$-2$} [t] at 25 5
 \pinlabel {$2$} [t] at 63 5
 \pinlabel {$2$} [l] at 47 27
 \pinlabel {$6$} [l] at 47 64
 \pinlabel {$12$} [t] at 153 5
\endlabellist

\begin{center}
\includegraphics{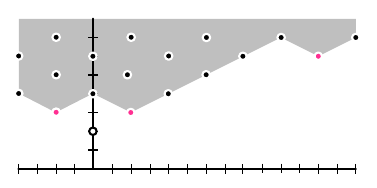}
\end{center}
\caption{The $3$-dimensional geography of the knot $P_3$ is generated by surface at $(12,6)$ and the positively and negatively twisted Seifert surface at $(\pm 2,3)$.  The Type I surface also lies at $(-2,3)$.}
\label{fig:P333-geography3}
\end{figure}

\begin{rem}
Hatcher and Oertel's algorithm goes further than just listing candidates the essential surfaces in the exterior of a Montesinos knot --- it can actually identify which of these are essential.  We chose not to use this part of Hatcher and Oertel's work since the combinatorics necessary to pick out the essential surfaces from the candidates were more involved than simply computing the normal Euler and Betti numbers for all of the candidates.  Nevertheless, computations using Dunfield's implementation of Hatcher and Oertel's algorithm \cite{dunfield:ho} indicates that the essential surfaces for $P_n$ correspond to the edgepath systems $\ep^I$, $\ep^{II}_0$ (the Seifert surface), $\ep^{III}_0$, and $\ep^{III}_3$.
\end{rem}

%% file: stable-gap-sum.tex
%!TEX root = non-ori-3d-4d.tex

We conclude by investigating the behavior of the $3$-dimensional spanning surface defect under connected sum. The proof of Theorem~\ref{thm:large-gap} will follow as an immediate corollary.

\begin{lem}\label{lem:connectedsum} 
	The $3$-dimensional spanning surface defects $\sg^\pm_3$ are each additive under connected sum, i.e.\ 
\[\sg^{\pm}_3(K_0\#K_1) = \sg^{\pm}_3(K_0) + \sg^{\pm}_3(K_1).\]
%	The stable non-orientable $4$-genera $\sg^\pm_4$ are subadditive under connected sum, i.e.\ 
%\[\sg^{\pm}_4(K_0\#K_1) \leq \sg^{\pm}_4(K_0) + \sg^{\pm}_4(K_1).\]
\end{lem}

The proof of the lemma is, in essence, the same as the proof of subadditivity of the non-orientable $3$-genus in \cite[Theorem 2.8]{clark:crosscap} with two important differences:
\begin{enumerate}
\item The need to track the normal Euler number, and
\item The freedom to add twisted bands to move off of the orientable spanning surfaces that force subadditivity rather than additivity of the ordinary non-orientable $3$-genus.
\end{enumerate}

In the proof below, we use the standard notation that if $F_i$ is a spanning surface for $K_i$, then $F_0\natural F_1$ is the boundary connected sum of $F_0$ and $F_1$, with $\partial(F_0\natural F_1) = K_0 \# K_1$. 

\begin{proof}
We write the proof in the positive case; the negative case is analogous. Let $F$ be a non-orientable spanning surface for $K_0 \# K_1$ that realizes $\sg^+_3(K_0 \# K_1)$.  

Using standard arguments (see, for example, \cite[\S5]{fox:quick-trip}), we may decompose $F$ as $F_0 \natural F_1$, where $F_i$ is a spanning surface for $K_i$.  If $F$ realizes $\sg^+_3(K_0 \# K_1)$, then so does $\twist^k_+(F)$, the surface obtained from the addition of $k$ positive twisted bands. Thus, we may assume that the surfaces $F_i$ are non-orientable.  It suffices to prove that, for $i=0,1$, we have $\sg^+_3(K_i) = \Gamma^+(F_i)$.

Suppose, to the contrary, that there exists a spanning surface $F_0^*$ of $K_0$ so that 
\begin{equation} \label{eq:cs-contra}
	\Gamma^+(F_0^*) < \Gamma^+(F_0).
\end{equation}
If $e(F_0^*) \leq e(F_0)$, then add $k= e(F_0)-e(F_0^*)$ positive twisted bands to $F_0^*$, yielding a surface $\twist^k_+(F_0^*)$ with $e(\twist^k_+(F_0^*)) = e(F_0)$.  As adding twisted bands does not change $\Gamma^+$, Equation \eqref{eq:cs-contra} shows that $b_1(\twist^k_+(F_0^*)) < b_1(F_0)$.  It follows that $b_1(\twist^k_+(F_0^*) \natural F_1) < b_1(F)$. Since $\twist^k_+(F_0^*) \natural F_1$ and $F$ have the same normal Euler number, we see that $\Gamma^+(\twist^k_+(F_0^*) \natural F_1) < \Gamma^+(F)$.  This contradicts the assumption that $F$ realizes $\sg^+_3(K_0 \# K_1)$.

If, on the other hand, $e(F_0^*) > e(F_0)$, we add twisted bands to $F_0$ so that $e(\twist^k_+(F_0)) = e(F^*_0)$.  A parallel argument to the one above once again leads to the required contradiction.
\end{proof}

\begin{proof}[Proof of Theorem~\ref{thm:large-gap}]
As in Section~\ref{sec:stable-gap-pretzel}, consider the pretzel knot $P=P(-3,3,3)$.  We know that $\sg^-_3(P) = 2$, while $\sg^-_4(P) = 0$ since $P$ is slice.  Let $K_n = \#_n P$.  On one hand, Lemma~\ref{lem:connectedsum} shows that $\sg^-_3(K_n) = 2n$.  On the other, the knot $K_n$ is slice, so $\sg^-_4(K_n) = 0$.  The theorem follows.
\end{proof}